\documentclass[11pt]{article}
\usepackage{amsmath}
\usepackage{url}
\usepackage{amsfonts}
\usepackage{amssymb}
\usepackage{amsthm}
\usepackage{graphicx}
\usepackage{psfrag}

\newtheorem{theorem}{Theorem}[section]
\newtheorem*{therem}{Theorem}
\newtheorem{proposition}[theorem]{Proposition}
\newtheorem{definition}[theorem]{Definition}
\newtheorem{lemma}[theorem]{Lemma}

\theoremstyle{remark}
\newtheorem{example}{Example}
\newtheorem{note}{{\bf Note}}
\newcommand{\spn}{\mathrm{span}}

\newcommand{\spam}{\mathop{\mathrm{span}}}

\newcommand{\nats}{\mathbb{N}}
\newcommand{\reals}{\mathbb{R}}
\newcommand{\RR}{\mathbb{R}}
\newcommand{\comps}{\mathbb{C}}
\newcommand{\M}{\mathbb{M}}
\newcommand{\calo}{\mathcal{O}}

\newcommand{\sphere}{\mathbb{S}^d}
\newcommand{\sph}{\mathbb{S}} 
\newcommand{\Exp}{\operatorname{Exp}}
\newcommand{\man}{\M}
\newcommand{\dif}{\mathrm{d}}

\newcommand{\vect}[1]{\mathbf{#1}}

\renewcommand{\b}{\mathbf{b}}

\newcommand{\bfA}{\mathbf{A}}

\newcommand{\bfy}{\mathbf{y}}
\newcommand{\inj}{\mathrm{r}_\M}

\renewcommand{\d}{\mathrm{dist}}

\newcommand{\J}{\mathcal{J}}

\renewcommand{\L}{\mathcal{L}}
\newcommand{\ns}[1]{\left|\!\left|\!\left| {#1}\right|\!\right|\!\right|}

\def\calc{{\mathcal C}}
\def\K{\mathrm{K}}
\def\coll{\K_{\Xi}}

\def\bfa{{\bf a}}
\def\bfb{{\bf b}}
\def\bfe{{\bf e}}

\def\bfy{{\bf y}}

\def\bfzero{{\bf 0}}
\def\achi{\widetilde{\chi}}
\numberwithin{equation}{section}

\numberwithin{equation}{section}

 \title{Better bases for kernel spaces}
 
 \author{E.~Fuselier\thanks{ Department of Mathematics, High Point University,
 High Point, NC 27262, USA.}, 
 T.~Hangelbroek\thanks{ Department of Mathematics, University of Hawaii, 
 Honolulu, HI 96822, USA.  Research supported by  by grant DMS-1047694 from the National
    Science Foundation.}, 
F. J.~Narcowich\thanks{ Department of Mathematics, Texas A\&M
    University, College Station, TX 77843, USA. Research
    supported by grant DMS-0807033 from the National
    Science Foundation.}, 
J. D.~Ward\thanks{ Department of Mathematics, Texas A\&M University,
    College Station, TX 77843, USA. Research supported by
    grant DMS-0807033 from the National Science
    Foundation.},
 G. B~Wright\thanks{ Department of Mathematics, Boise State University, Boise, ID 83725, USA.
    Research supported by grants DMS-0934581and DMS-0540779 from the National Science
    Foundation.}   
    }

\begin{document}
\maketitle
\begin{abstract}
 In this article we investigate the feasibility of constructing stable, local bases for computing with kernels. In 
 particular, we are interested in constructing families $\bfb = (b_{\xi})_{\xi\in\Xi}$ that function as bases for
 kernel spaces $S(k,\Xi) = \{ \sum_{\xi\in\Xi} a_{\xi} k(\cdot,\xi)\mid (a_{\xi})_{\xi\in\Xi} \in \reals^{\Xi} \}$ so
 that each basis function can be 
 obtained by very few kernels
$$b_{\zeta} =  \sum_{\xi\in\Xi} A_{\zeta,\xi} k(\cdot,\xi) \qquad A_{\zeta,\xi} = 0\text{ for all but a few } \xi.$$
This is reminiscent of the construction of the B-spline basis from the family of  truncated power functions. 

We demonstrate that for a large class of kernels (the Sobolev kernels as well as many kernels of polyharmonic and related type) such bases exist . In fact, the basis elements can be constructed using a combination of roughly $\mathcal{O}(\log N)^d$ kernels, where $d$ is the local dimension of the manifold and $N$ is  the dimension of 
the kernel space (i.e.  $N=\#\Xi$). Viewing this as a preprocessing step -- the construction
of the basis has computational cost $\mathcal{O}\bigl(N(\log N)^d\bigr)$.
Furthermore, we prove that the 
new basis is $L_p$ stable and satisfies polynomial decay estimates that are stationary with respect
to the density of $\Xi$.
\end{abstract}


\section{Introduction}\label{sec0}

The purpose of this article is to investigate \emph{robust} bases for
spaces associated with a positive definite or conditionally positive
definite kernel $k:\M\times\M \to \RR$, where $\M$ is a $C^\infty$
(closed) compact Riemannian manifold. The dimension of $\M$ is
$d$. The kernels that we discuss below belong to a wide class that
includes the thin-plate splines and similar kernels when $\M=\sph^d$ or
$SO(3)$.

The spaces associated with a kernel from this class are defined as
follows. Let $\Xi \subset \M$ be a finite set of points, called
centers, having cardinality $N=\# \Xi$. The centers are
\emph{scattered} in the sense that they do not need to belong to a
regular grid. In the positive definite case, the space $S(k,\Xi)$
associated with the kernel $k$ and the set $\Xi$ is just
$S(k,\Xi)=\spn\{k(\cdot,\Xi), \ \xi\in \Xi\}$; that is,
\[
S(k,\Xi) := \left\{\sum_{\xi\in \Xi} a_\xi k(\cdot,\xi),\
  a_\xi\in \RR \right\}.
\]
The conditionally positive definite case is similar; we will discuss it
it in Section \ref{interpolation_kernels} -- specifically in (\ref{space_def}). 
For these spaces, if $h:=
 \max_{\xi \in \Xi}
\text{dist}(x,\xi)$ is the mesh norm (fill distance), $q:=
\frac12 \min_{\xi\ne \eta} \text{dist}(\eta,\xi)$ is the separation radius and
$\rho_\Xi := 
h/q$ is the mesh ratio, then as long as $\rho_\Xi
\le \rho_0$, where $\rho_0$ is fixed, Lebesgue constants are uniformly
bounded and approximation rates for functions in Sobolev spaces
$W_p^m(\M)$ are $\calo (h^m)$, with the constants independent of other
properties of $\Xi$ \cite{HNW,HNW2}.

Two other remarkable properties of $S(k,\Xi)$ concern its Lagrange
basis, $\{\chi_\xi(\cdot)\}_{\xi\in \Xi}$. Recall that in
a Lagrange basis each basis function satisfies $\chi_\xi(\eta) =
\delta_{\xi,\eta}$ when $\eta \in \Xi$. What was shown in \cite{HNW,HNW2}
is that $\chi_\xi$ decays \emph{exponentially} fast away from $\xi$
for special kernels, and \emph{algebraically} fast for many
others.

Equally as important, as we shall prove below in Theorem \ref{main}, 
if we express the $\chi_\xi$'s in the standard basis,
\[
\chi_\xi = \sum_{\eta\in \Xi} A_{\xi,\eta}k(\cdot, \eta),
\]
where the coefficients $A_{\xi,\eta}$ are well-known to be the entries
of the inverse of the interpolation matrix, then $|A_{\xi,\eta}|$
decays as a function of $\text{dist}(\eta,\xi)$ at the same rate as
$|\chi_\xi(x)|$ decays in $\text{dist}(x,\xi)$ -- \emph{i.e.},
exponentially or algebraically, as the case may be. Prior to our work,
the only provable results concerning decay of these coefficients were
done by Fornberg \cite{Forn} in the case of $\RR$ and $\RR^2$ for
gridded data, using Fourier techniques that do not carry over to the
scattered case.

The difficulty with the Lagrange basis is that each $\chi_\xi$ is
computationally costly both to construct (as a linear combination of $k(\cdot,\xi)$, $\xi\in \Xi$) 
and to compute with. Are there better bases? Here is what we would desire in a basis 
$\{b_\xi\}_{\xi\in  \Xi}$ for $S(k,\Xi)$.

Each basis function should be highly localized and nearly scalable
with respect to the mesh norm $h$ of $\Xi$. By this we mean that each
basis element is of the form
\[
b_\xi  = \sum A_{\xi,\eta}k(\cdot,\eta)
\]
where the $\eta$'s come from small subset of the centers $\Xi$ and
satisfy the following requirements:
\begin{align*}
  \text{\phantom{i}i)}~~ &\#\{A_{\rho,\xi}\ne 0\} = \mathrm{c}(\# \Xi)\\
  \text{ii)}~~ &|b_\rho(x)| \le
  \sigma\left(\frac{\text{dist}(x,\xi)}h\right)
\end{align*}
where the cost $\mathrm{c}(N)$ is constant or slowly growing with $N = \#\Xi$
and the function
$\sigma(x)$ decays rapidly: 
at an exponential rate $\sigma(x) \le
Ce^{-\nu|x|}$ or 
at least at a fast polynomial rate $\sigma(x) \le
C(1+|x|)^{-J}$. The B-spline basis, constucted from the
family of truncated power functions, is a model kernel method and in
the $d=1$ case provides an ideal solution to the problem we consider.

The main results of this paper demonstrate that such kernel bases
exist and that each basis function can be computed in nearly fixed
time. Viz.,\ an individual basis function can be computed in
${\calo}((\log N)^d)$ time while the full basis can be computed in
${\calo}(N(\log N)^d)$ time with $N=\# \Xi = Ch^{-d}$. Moreover the
basis is $L^p$ stable. 

The main tool employed is Theorem \ref{main} which allows one to
bound the rate of decay of Lagrange coefficients in terms of the
corresponding decay rate of the Lagrange functions. In particular if
the Lagrange functions have exponential decay so too do their
corresponding coefficients. While, as we mentioned earlier, this fact
was previously known in the scaled lattice case, but no such estimates
have been available in the scattered case.

{\bf The sphere $\mathbb{S}^2$} As an example of our main results for the sphere $\mathbb{S}^2$
and the {\em restricted surface splines} of order $s+1$, given by
$k_{s+1}(x,\alpha) := (1- x\cdot \alpha)^{s}\log(1-x\cdot\alpha)$,
 for $s = 1,2,3,\dots$, 
 we have the following theorem, which is a corollary of Theorem \ref{better_theorem} in Section \ref{S:better_basis}.
\begin{therem}
For 
a sufficiently dense  set of centers $\Xi$,  
and for a sufficiently large constant $\tau$
there is basis $(b_{\xi})_{\xi\in\Xi}$ for the space $S(k,\Xi)$ satisfying the following. 
\begin{itemize}
\item Each basis element
$b_{\xi} = \sum_{\zeta \in\Xi} A_{\xi,\zeta} k(\cdot,\zeta)$ is  composed of at most
$M := \tau(\log N)^2$ kernels. I.e.
$$
 \mathrm{c}(\# \Xi)= M = \tau \left(\log N \right)^2.$$
\item
Each basis element exhibits polynomial decay: there exist  constant $C$ and $J$
for which
 $$|b_{\xi}(x)|\le C \left(1+\frac{\d(x,\xi)}{h}\right)^{-J}.$$
\item The rate of polynomial decay $J$ depends linearly on the constant of proportionality $\tau$ by
$$ J=\mathcal{O}( \sqrt{\tau}).$$
 \item The basis is $L_p$ stable: there are constants $0<c_1\le c_2<\infty$ depending on $\tau$ 
 so that for all sequences
$\bfa = (a_{\xi})_{\xi\in\Xi}\in \reals^{\Xi}$ the following holds:
$$c_1 q^{2/p} \|\bfa\|_{\ell_p(\Xi)}
\le 
\left\|\sum_{\xi\in\Xi} a_{\xi} b_{\xi}\right\|_{L_p( \mathbb{S}^2 )}
\le 
c_2 q^{2/p} \|\bfa\|_{\ell_p(\Xi)}.$$
\end{itemize}
\end{therem}
The restricted surface splines on $\mathbb{S}^2$ are of special importance, largely because 
the sphere is the setting of many problems of scientific interest, but also because the kernels themselves
have a convenient, closed form representation and their approximation power that is well understood
and optimal in the sense that approximation rates are in line with smoothness assumptions
for the target functions (i.e., approximands in Sobolev classes $W_p^s$ or Besov classes $B_{p,\infty}^s$ 
are approximated by functions in $S(k,\Xi)$ with error decaying like $h^s$). 

We give some numerical examples for $\mathbb{S}^2$  with such kernels (and others) in Section \ref{S:Examples}.

{\bf Precondioners} Over the years practical implementation of kernel approximation has
progressed despite the ill-conditioning of kernel bases. This has
happened with the help of clever numerical techniques like multipole
methods \cite{Greengard, BPT, BL} and often with the help of preconditioners 
\cite{BGP, FGP,Kansa,Sloan}
 of which  \cite{BP, CR1, CR2}  are of special interest to us, because
these involve attempts to construct local bases.
Indeed, another offshoot of our results in that one can now estimate, a
priori, the number of coefficients needed to guarantee good
preconditioners. Many results already exist in the RBF literature
concerning preconditioners and ``better'' bases. For a good list of
references and further discussion, see \cite{Fasshauer}. Several of these papers use ``local
Lagrange'' functions in their efforts to efficiently construct
interpolants. The number of points chosen to localize the Lagrange
functions are ad hoc and seem to be based on experimental
evidence. For example, Faul and Powell, in \cite{FaulPowell},  devise an algorithm which converges
to a given RBF interpolant that is based on local Lagrange interpolants
using about thirty nearby centers. Beatson--Cherrie--Mouat, in \cite{BCM},
use fifty local centers (p.~260, Table~1) in their construction along
with a few ``far away'' points to control the growth of the local
interpolant at a distance from the center. In other work, Ling and Kansa \cite{Kansa}
and co-workers have studied approximate cardinal basis functions based
on solving least squares problems. Thus one goal of this paper is to
provide some theoretical groundwork that may yield future improvements
in preconditioner algorithms and better bases for kernel spaces.

{\bf Organization} We devote Section \ref{S:Background} to treating 
some pertinent results and definitions for Riemannian
manifolds.   
In Section \ref{S:LagrangeBasis} we consider the stable, local
bases constructed in \cite{HNW,HNSW,HNW2}, which have many desirable
properties but are computationally infeasible due to their cumbersome
construction---each basis function of this type requires $\#\Xi$
nonzero kernel coefficients in its construction and, moreover,
computing these requires ${\calo}((\#\Xi)^3)$ operations. An analysis
of these coefficients show that they drop off rapidly---this is
demonstrated in Section \ref{S:Coefficients}. The rapid decay of these
coefficients leads to the (theoretical) existence of efficiently
constructed bases, but sadly does not indicate the desired
construction -- this is treated in Section \ref{S:better_basis}. In Section \ref{S:Examples}, we give numerical evidence to bolster the
results of the previous sections, by giving results of  experiments that
show how rapidly the Lagrange basis and the coefficients decay. In
this section we give some examples of techniques that fail to deliver,
and provide some examples of families that seem to have the desired
properties which have not been validated theoretically.

\section{Geometric background}\label{S:Background}
Throughout this paper, $\M$ denotes a compact, complete $d$-dimensional Riemannian manifold.
The Riemannian metric for $\M$ is $g$, which defines
an inner product $g_p(\cdot,\cdot )=\langle\cdot,\cdot \rangle_{g,p}$
on each tangent space $T_pM$; the corresponding norm is
$|\cdot|_{g,p}$.

The Riemannian metric is employed to measure arc length of a curve $\gamma$ via $\int_a^b |\dot
\gamma|_{g,p}dt$. Geodesics are curves $\gamma : \RR \to \M$ that locally
minimize the arc length functional giving rise to a distance function 
$$\d(p,q)  = \min_{\substack{\gamma(0) = p\\\gamma(1) = q}}
\int_0^1 |\dot\gamma|_{g,p}dt.$$
We denote the ball in $\M$ centered at $x$ having radius $r$ by $B(x,r).$ 
Given a finite set $\Xi\subset\M$, we define its  \emph{mesh norm}  (or \emph{fill distance}) $h$ 
and the \emph{separation radius} $q$ to be:
\begin{equation} \label{minimal-separation}
 h:=\sup_{x\in \M} \d(x,\Xi)\qquad \text{and}\qquad  q:=\frac12 
\inf_{\xi,\zeta\in \Xi, \xi\ne \zeta} 
\d(\xi,\zeta).
\end{equation}
The mesh norm measures the density of $\Xi$ in $\M$, the separation radius determines 
the spacing of $\Xi$. The \emph{mesh ratio} $\rho:=h/q$ measures the uniformity of the 
distribution of $\Xi$ in $\M$.  
We say that the point set $\Xi$ is quasi-uniformly distributed, or simply that $\Xi$ is quasi-uniform if 
$\Xi$ belongs to a class of finite subsets with mesh ratio bounded by a constant $\rho_0$.

The metric $g$ also induces an invariant volume measure $d\mu$ on
$\M$. The local form of the measure is $d\mu(x) =
\sqrt{\det(g)}dx^1\cdots dx^d$, where $\det(g) =
\det(g_{ij})$. We indicate the measure of
subsets $\Omega\subset \M$ by $\mathrm{vol}(\Omega)$. The integral, and the $L_p$ 
spaces for $1\le p\le \infty$, are defined with respect to this measure.  The embeddings
$$C(\M) \subset L_p(\M)\ \text{for}\ 1\le p \le \infty 
\quad \text{and} \quad 
L_p(\M) \subset L_q(\M)\ \text{for}\ 1\le q\le p\le \infty$$ 
hold. In addition, $L_2$ is a Hilbert space equipped with the inner product 
$\langle\, \cdot\, ,\, \cdot \,\rangle\colon \ (f,g)\mapsto \langle f,g\rangle$

\paragraph{Sobolev spaces on subsets of $\M$}
Sobolev spaces on subsets of a Riemannian manifold can be defined in
an invariant way, using the covariant derivative (or connection) $\nabla$ (cf. \cite{Aub}) 
which maps tensor fields of 
rank $j$ to tensor fields of rank $j + 1$. The $k$th covariant
derivative of a function is a rank $k$ tensor field and is denoted $\nabla^k f$. 
For $k=1$,  the covariant derivative in local
coordinates is simply the usual expression for the ``gradient'' -- it 
can be written simply as $(\nabla f(x))_j = \frac{\partial f}{\partial x^j} f(x)$.
For $k=2$, the ``Hessian'' tensor involves 
Christoffel symbols $\Gamma_{ij}^{m}$ and can be expressed  as 
$(\nabla^2 f(x))_{i,j} = \frac{\partial^2f}{\partial x^i x^j} (x) - \sum_{m=1^d} 
\Gamma_{i,j}^{m}(x) \frac{\partial f}{\partial x^m}(x) $.
Higher order covariant derivatives have an analogous expression, using higher order
derivatives of the Christoffel symbols -- see \cite[Eqn. (3)]{HNW}.

\begin{definition}[{\cite[p.~32]{Aub}}]\label{Sob_Norm}
  Let $\Omega\subset \man$ be a measurable subset. We define the Sobolev 
  space $W_2^m(\Omega)$ to be all $f:\M \to \reals$ such that, for $0\le k\le m$, 
  $ | \nabla^k f|_{g,p} $ in $L_2(\Omega)$ with associated norm
%
%
\begin{equation}\label{def_spn} 
\|f\|_{m,\Omega}^2 :=\|f\|_{W_2^m(\Omega)}:= 
  \left( 
    \sum_{k=0}^m
    \int_{\Omega} 
    | \nabla^k f |_{g,p}^2
    \, \dif \mu(p)\right)^{1/2}, 
\end{equation}
coming from the Sobolev inner product
\begin{equation}\label{def_sn} 
\langle f, g\rangle_{m,\Omega}:=\langle f,g\rangle_{W_2^m(\Omega)}:= 
  \sum_{k=0}^m
  \int_{\Omega} 
    \left\langle
      \nabla^k f,  \nabla^k g
    \right\rangle_{g,p}
  \, \dif \mu(p).
\end{equation}
 When $\Omega=\man$, we may suppress
the domain: $\langle f,g\rangle_m = \langle f,g\rangle_{m,\M}$ and
$\|f\|_m=\| f\|_{m,\M}$.
\end{definition}
\paragraph{Metric equivalence} The exponential map allows us to
compare the Sobolev norms we've just introduced, to standard Euclidean
Sobolev norms as follows:
\begin{lemma}[{\cite[Lemma~3.2]{HNW}}]\label{Fran}
  For $m\in\nats$ and $0<r<\inj/3$, there are constants $0 < c_1 <c_2$
  so that for any measurable $\Omega \subset B_{r}$, for all $j\in
  \nats$, $j\le m$, and for any $p_0\in \M$, the equivalence
$$
c_1\| u\circ\Exp_{p_0}\|_{W_2^j(\Omega)}\le
\|u\|_{W_2^j(\Exp_{p_0}(\Omega))}\le c_2\|
u\circ\Exp_{p_0}\|_{W_2^j(\Omega)}
$$
holds for all $u:\mathrm{\Exp}_{p_0}(\Omega)\to \reals$. The constants
$c_1$ and $c_2$ depend on $r$ and $m$ but they are
\emph{independent} of $\Omega$ and $p_0$.
\end{lemma}

\section{The Lagrange basis}\label{S:LagrangeBasis}
For a manifold $\M$, a  positive definite kernel $k:\M\times\M\to \reals$ and a set of centers $\Xi\subset\M$, 
we are concerned with the robustness of the Lagrange basis 
$(\chi_{\zeta})_{\xi\in\Xi}$ for $S(k,\Xi)$, where 
$\chi_{\zeta}(\xi)= \delta_{\xi,\zeta} $ for all $\xi \in \Xi$. 
The Lagrange basis plays a central role in most interpolation problems, and certainly
this is the case
for radial basis function and kernel interpolation. Decay of the Lagrange basis and
analytic consequences have notably been considered in \cite{Madych, buhmann_1990, Siva, Forn}.

\subsection{The kernels considered}
More recently, \cite{HNW,HNSW, HNW2},  develop a theory for fast decay and stability of the Lagrange basis associated with certain positive definite and conditionally positive definite kernels.
\begin{itemize}
\item  ``Sobolev kernels'' denoted by $\kappa_m$ were introduced in \cite{HNW} 
for any smooth, complete, compact and connected Riemannian manifold. These are the 
reproducing kernels for the Sobolev inner product\footnote{In fact, the inner product can be weighted as 
$ \sum_{k=0}^m
c_k
  \int_{\Omega} 
    \left\langle
      \nabla^k f,  \nabla^k g
    \right\rangle_{g,p}
  \, \dif \mu(p)$ with non-negative weights $c_k$ for which $c_0$ and $c_m$ are postive -- each
  such reweighting gives a different inner product and a different Sobolev spline.} for $W_2^m(\M)$ when $m>d/2$:
$$(u,v) \mapsto \langle u,v\rangle_{W_2^m(\Omega)} =   \sum_{k=0}^m
  \int_{\Omega} 
    \left\langle
      \nabla^k f,  \nabla^k g
    \right\rangle_{g,p}
\,   \dif \mu(p).$$ 
\item This was extended in \cite{HNW2} to treat a broader class of  kernels on certain manifolds
called kernels of polyharmonic
and related type (these are discussed in Section \ref{S:Estimating}).
\item Included in the class considered in \cite{HNW2} are 
{\em restricted surface splines} on $\sphere$ which are kernels of the form $k_m(x,\alpha)  = \phi(x\cdot \alpha)$ where
$$ \phi(t) = \begin{cases}(1-t)^{m-d/2} &\text{ for } d \text{ odd }\\ (1-t)^{m-d/2} \log(1-t) &\text{ for } d \text{ even}.\end{cases}$$
 These kernels are conditionally positive definite, meaning that interpolants are constructed
by adding an auxiliary function. (In this case, a low degree spherical harmonic.) See Section \ref{interpolation_kernels} below.

The expansion of the functions $\phi$ in terms of Gegenbauer polynomials by Baxter
and Hubbert, \cite{BaHu}, leads to Fourier (spherical harmonic) expansions of the kernels,
and from there to their characterization as Green's functions for elliptic differential operators. These
operators are of polyharmonic type -- they are of the form $Q(\Delta) = \prod_{j=1}^m (\Delta - r_j)$
for some real numbers $r_1, \dots, r_m$.
This, in turn, permits an understanding of the approximation power of the kernel, as investigate in 
\cite{MNPW, Hsphere}:
for functions having $L_p$ smoothness $s$ up to order $2m$ (namely, for target functions in 
smoothness spaces including $B_{p,q}^s(\sphere),W_p^s(\sphere),C^s(\sphere)$ with $s\le 2m$), 
 $$\d_p\bigl(f,S(k,\Xi)\bigr)
  = \mathcal{O}(h^s).$$
Here, the space $S(k,\Xi)$ is modified by addition of low degree spherical harmonic terms
$\Pi = \{Y_{\ell,m}\mid \ell\le \lfloor m-d/2\rfloor \}$ (this is described in Section \ref{space_def} below).

\item Surface splines on $SO(3)$ which are of the form $k(x,\alpha)  = \phi(\omega(\alpha^{-1}x) )$ 
with $\omega(x)$ the angle of rotation of $x$ (which is a left and right invariant metric on 
the group) and
$$ \phi(t) =  \bigl(\sin(t/2)\bigr)^{m-3/2}.$$
In \cite{HS}, an expansion of $\phi$ in  even Chebyshev polynomials of the second kind leads to
a Fourier (Wigner D-function) expansion of the kernel $k$. As in the spherical case,
this leads to its characterization as a Green's function for an operator of polyharmonic
type on $SO(3)$, and to a realization of its approximation power:
 again, for $f\in B_{p,q}^s(SO(3)),W_p^s(SO(3)),C^s(SO(3))$ with  $s\le 2m$ we have
 $$\d_p\bigl(f,S(k,\Xi)\bigr) = \mathcal{O}(h^s).$$ 
\end{itemize}

{\bf Restricted kernels} An alternative approach, taken in \cite{FuWr}, is to consider the manifold
$\M$ as embedded in an ambient Euclidean space $\reals^n$, and to use the restriction of 
a radial basis function -- a Euclidean (conditionally) positive definite kernel satisfying 
rotational symmetry (of which there are many prominent examples) --  as a (conditionally)
positive definite kernel on $\M$. In a sense, this is a completely different approach, in the sense
that such kernels are almost never fundamental solutions to differential operators, a key point 
of \cite{HNW2}. On the other hand, such kernels may be easily localized in the ambient space $\reals^n$,
which may lead to an effective way of localizing and preconditioning the restricted kernels. Although
the theory developed in Sections 3 and 4 does not address such kernels, we include a numerical
example in Section \ref{S:Examples}.

\subsection{Analytic properties of the Lagrange basis} The theory developed in \cite{HNW,HNSW,HNW2} addresses analytic properties of bases for 
$S(k,\Xi)$,  related to locality, stability of approximation and interpolation.
In particular the following are shown.

{\bf Locality}. The Lagrange basis is a local bases for $S(k,\Xi)$. That is, 
$$|\chi_{\xi}(x)| \le C \exp\left(-\nu \frac{\d(x,\xi)}{h}\right).$$

{\bf Stability of interpolation}. Interpolation is stable: the
Lebesgue constant is bounded (and more generally, is the $p$ norm of the
interpolant  is controlled by the $\ell_p$ norm of the data).

{\bf $L_p$ conditioning}.
There are constants depending only on $c_1,c_2$ such that 
$c_1 \|a\|_{\ell_p}\le\|\sum_{j=1}^N a_{\xi} \chi_{\xi}\|_{L_p} \le c_2 \|a\|_{\ell_p}$, with $c_1,c_2$
depending only on $m$, $\M$ and the mesh ratio $\rho$.
In particular, they are independent of $\#\Xi = \dim(S(k,\Xi))$, and, after a suitable normalization, independent of
$p$.

{\bf Marcinkiewicz-Zygmund property}. The space $S(k,\Xi)$
possess a Marcinkiewicz-Zygmund property relating samples to the size
of the function.  For $s\in S(k,\Xi)$, this means that the norms
$\|\xi\mapsto s(\xi)\|_{\ell_p(\Xi)}$ and $\|s\|_{L_p}$ are equivalent, with
constants involved independent of $\#\Xi$.

{\bf Stability of approximation in $L_p$}. Approximation by $L_2$
projection is stable in $L_p$ for $1\le p\le \infty$.  In particular,
 the orthogonal projector with range $S(k,\Xi)$ can be
continuously extended to each $L_p$, and it has bounded operator
norm independent of $\#\Xi$.

%
%

\section{Lagrange function coefficients}\label{S:Coefficients}
In this section we give theoretical results for the coefficients
in the kernel expansion of  Lagrange functions. In the first
part we give a formula, relating these coefficients to native
space inner products of the Lagrange functions themselves
(this is  Proposition \ref{Lagrange_Coeffs_Formula}). We then
obtain estimates on the decay of these coefficients for a 
class of  kernels on certain compact Riemannian manifolds 
(two point homogeneous spaces).   

\subsection{Interpolation with conditionally positive definite
  kernels}
  \label{interpolation_kernels}
The kernels we consider in this article are conditionally positive
definite on the compact Riemannian manifold. As a reference on this
topic, we suggest \cite[Section 4]{DNW}.

\begin{definition}\label{cpd}
A kernel is conditionally positive definite with
  respect to a finite dimensional space $\Pi$ if, for any set of
  centers $\Xi$, the matrix $\calc_{\Xi}  := \bigl(k(\xi,\zeta)
  \bigr)_{\zeta,\xi\in\Xi}$ is positive definite on the subspace of
  all vectors $\alpha\in \comps^{\Xi}$ satisfying $\sum_{\xi\in\Xi}
  \alpha_{\xi} p(\xi) = 0 $ for $p\in \Pi$.
\end{definition}

This is a very general definition which we will make concrete in the
next subsections. Given a complete orthonormal basis
$(\phi_j)_{j\in\nats}$, of continuous functions
(i.e., $\|\phi_j\|_{\infty}= 1$) any kernel
$$k(x,y):=\sum_{j\in \nats} \tilde{k}(j) \varphi_j(x)\overline{\varphi_j(y)}$$
with coefficients $\tilde{k}\in \ell_2(\nats)$ for which all but
finitely many coefficients $\tilde{k}(j)$ are positive (negative) is
conditionally positive definite with respect to $\Pi_{\J} =
\spam(\phi_j \mid j\in \J),$ where $\J = \{j\mid \tilde{k}(j) \le
0\},$ since, evidently,
\begin{eqnarray*}
  \sum_{\xi\in\Xi} \sum_{\zeta\in\Xi}  \alpha_{\xi}
  k(\xi,\zeta)\overline{
    \alpha_{\zeta}}
  &=&
  \sum_{\xi\in\Xi} \sum_{\zeta\in\Xi}  \alpha_{\xi} \overline{\alpha_{\zeta}}
  \left( \sum_{j\in\nats} \tilde{k}(j) \phi_j(\xi)
    \overline{\phi_j(\zeta)} 
  \right) \\
  &=& 
  \sum_{j\in\nats}\tilde{k}(j)\sum_{\xi,\zeta\in\Xi} \alpha_{\xi} 
  \phi_j(\xi)\overline{ \alpha_{\zeta}\phi_j(\zeta) } =
  \sum_{j\notin \J} \tilde{k}(j) \|\alpha\phi_j\|_{\ell_2 (\Xi)}^2 >0
\end{eqnarray*}
provided $\sum_{\xi} \alpha_{\xi} \phi_j(\xi) = 0$ for $j$ satisfying
$\tilde{k}(j)\le 0$.

In this case if the set of centers $\Xi\subset \M$ is unisolvent with
respect to $\Pi_{\J} = \spam(\varphi_j\mid j\in\J)$ (meaning that
$p\in \Pi_{\J} $ and $p(\xi) = 0$ for $\xi \in\Xi$ implies that $p=0$)
then the system of equations
 \begin{equation}\label{Collocation_System}
 \left\{
\begin{array}{ll}
  \sum_{\xi\in\Xi} a_{\xi} k(\zeta,\xi) + \sum_{j\in\mathcal{J}} b_j 
  \varphi_j(\zeta)= y_{\zeta}&\quad \zeta\in\Xi\\
  \sum_{\xi\in\Xi} a_{\xi} \overline{\varphi_j(\xi)} = 0&\quad j\in\J
\end{array}
\right.
\end{equation}
has a unique solution in $\comps^{\Xi} \times\comps^{\J}$ for each
data sequence $\bfy = \bigl(y_{\zeta}\bigr)_{\zeta\in \Xi}\in \comps^{\Xi}$.

By writing the same system in matrix form, with collocation matrix 
$\coll = \bigl(k(\xi,\zeta)\bigr)_{\xi,\zeta\in \Xi^2}$ and auxiliary matrix 
$\Phi = \bigl(\phi_j (\xi)\bigr)_{(\xi,j)\in \Xi\times \J}$.
\begin{equation}
\begin{pmatrix}
\coll & \Phi\\
\Phi^* &0
\end{pmatrix}
\begin{pmatrix}
\bfa\\\bfb\end{pmatrix}
=
\begin{pmatrix}
\bfy\\\bfzero\end{pmatrix}
\end{equation}

When data is sampled from a continuous function at points $\Xi$
(i.e., $y_{\zeta} = f(\zeta)$) that are unisolvent\footnote{Meaning that $p|_{\Xi}=\bfzero$ for $p\in\Pi_{\J}$ implies
that $p = 0$.}
for $\Pi_{\J}$ this solution generates a continuous
interpolant:
$$
I_{\Xi}f =I_{k,\J,\Xi} f = \sum_{\xi\in\Xi} a_{\xi} k(\cdot,\xi) +
p_f
$$
where $p_f = \sum_{j\in\mathcal{J}} b_j \varphi_j\in\Pi_{\J}$ and $\sum_{\xi\in\Xi}a_{\xi}p(\xi) =0$
for all $p\in \Pi_{\J}$. Indeed, this interpolant is unique among functions from the space
\begin{equation}\label{space_def}
S(k,\Xi):= S(k,\Xi,\Pi_{\J}) : = \left\{ 	\sum_{\xi\in\Xi} a_{\xi} k(\cdot,\xi) +p_f \;\middle\vert\; p_f \in\Pi_{\J}, \quad \sum_{\xi\in\Xi}a_{\xi}p(\xi) =0	\right\}
\end{equation}
It has a dual role as  the minimizer of the  semi-norm $\ns{\cdot}_{k,\J}$ induced in 
the usual way from the ``native space'' semi-inner product
\begin{equation}\label{NS_norm}
\left \langle u,v\right\rangle_{k, \J} =
\left \langle  \sum_{j\in\nats}\hat{u}(j) \varphi_j, 
 \sum_{j\in\nats}\hat{v}(j) \varphi_j 
\right \rangle_{k, \J}
= \sum_{j\notin \J} \frac{\hat{u}(j) \overline{\hat{v}(j)} }{\tilde{k}(j)}.
\end{equation}
 
When $u, v\in S(k,\Xi)$ -- meaning that they  have the expansion
$
u=  \sum_{\xi\in{\Xi}} a_{1,\xi} k(\cdot,\xi) + p_u
$ 
and 
$v=  \sum_{\xi\in{\Xi}} a_{2,\xi} k(\cdot,\xi) + p_v
$
with coefficients $(a_{j,\xi})_{\xi\in\Xi} \perp (\Pi_{\J})|_{\Xi}$ for $j=1,2$ --
then the semi-inner product is
\begin{equation*}
\left \langle u,v\right\rangle_{k, \J} =
\sum_{\xi\in\Xi}\sum_{\zeta\in\Xi} a_{1,\xi}\overline{a_{2,\zeta}} k(\xi,\zeta)
\end{equation*}
We can use this expression of the inner product to investigate the kernel expansion
of the Lagrange function.
\begin{proposition} \label{Lagrange_Coeffs_Formula}
Let $k = \sum_{j\in \nats} \tilde{k}(j) \varphi_j\overline{\varphi_j}$ be a conditionally
positive definite kernel with respect to the space $\Pi_{\J}=\spam_{j\in\J} \varphi_j$, and
let $\Xi$ be unisolvent for $\Pi_{\J}$.
 Then $\chi_{\eta}\in S(k,\Xi)$ (the Lagrange function  centered at $\eta$)
has the kernel expansion 
$\chi_{\eta}(x) = \sum_{\xi\in\Xi} A_{\eta,\xi} k(\cdot,\xi) +
p_{\zeta}$
with coefficients $$\bfA_{\eta} = (A_{\eta,\xi})_{\xi\in\Xi} = 
\bigl(\langle \chi_{\zeta}(x), \chi_{\eta}(x)\rangle_{k, \J}\bigr)_{\xi\in\Xi}.$$
\end{proposition}
\begin{proof}
Select two centers $\zeta,\eta\in\Xi$ with corresponding Lagrange functions $\chi_{\zeta}$ and 
$\chi_{\eta}\in S(k,\Xi)$. 
Because $\bfA_{\zeta}$ and $\bfA_{\eta}$ are both orthogonal to $(\Pi_{\J})|_{\Xi}$, we have
$$\langle \chi_{\zeta}, \chi_{\eta}\rangle_{k, \J} = 
\sum_{\xi_1\in\Xi}\sum_{\xi_2\in\Xi} A_{\zeta,\xi_1}\overline{A_{\eta,\xi_2} }k(\xi_1,\xi_2)
=
\langle \coll \bfA_{\zeta},\bfA_{\eta}\rangle_{\ell_2(\Xi)}.$$

Now define $P:= \Phi (\Phi^* \Phi)^{-1} \Phi^*:\ell_2(\Xi)\to (\Pi_{\J})|_{\Xi}\subset \ell_2(\Xi)$ to be the orthogonal projection onto the subspace of samples of $\Pi_{\J}$ on $\Xi$
and let $P^{\perp} = \mathrm{Id}-P$ be its complement. 
Then for any data $\bfy$,
(\ref{Collocation_System}) yields coefficient vectors $\bfA$ and $\bfb$  satisfying
$P^{\perp} \bfA = \bfA$ and $P^{\perp} \Phi \bfb = \bfzero$, hence
$P^{\perp} \coll P^{\perp} \bfA = P^{\perp}\coll\bfA = P^{\perp} \bfy$. 
Because $P^{\perp}:\ell_2(\Xi)\to\ell_2(\Xi)$ is also an orthogonal projector, and therefore self-adjoint, it follows that
\begin{eqnarray*}
\langle \chi_{\zeta}(x), \chi_{\eta}(x)\rangle_{k, \J}  
&=&
\langle \coll \bfA_{\zeta},\bfA_{\eta}\rangle_{\ell_2(\Xi)}
=
\langle \coll  \bfA_{\zeta},P^{\perp}\bfA_{\eta}\rangle_{\ell_2(\Xi)}
=
\langle P^{\perp}\coll \bfA_{\zeta},\bfA_{\eta}\rangle_{\ell_2(\Xi)}\\
&=&
\langle P^{\perp}\bfe_{\zeta},\bfA_{\eta} 
\rangle_{\ell_2(\Xi)}.
\end{eqnarray*}
In the last line, we have introduced the sequence $\bfe_{\zeta} = (\delta_{\zeta,\xi})_{\xi\in\Xi}$ for
which $\coll \bfA_{\zeta} +p_{\zeta}|_{\Xi}=\bfe_{\zeta}$ which implies that  $P^{\perp}\coll \bfA_{\zeta} = P^{\perp} \bfe_{\zeta}$. Using once more the fact that $P^{\perp}$ is self-adjoint, and that $\bfA_{\eta}$ is in its
range, we have
$$\langle \chi_{\zeta}(x), \chi_{\eta}(x)\rangle_{k, \J}  =
\langle P^{\perp}\bfe_{\zeta},\bfA_{\eta}\rangle = \langle\bfe_{\zeta},  P^{\perp}\bfA_{\eta}\rangle
= \langle \bfe_{\zeta}, \bfA_{\eta}\rangle 
$$
and the lemma follows.
\end{proof}
\subsection{Estimating Lagrange function coefficients}\label{S:Estimating}
In \cite{HNW, HNW2}, it has been shown that Lagrange functions decay
rapidly away from the center. We can use this characterization of the Lagrange 
function to estimate the decay of its coefficients. In this section we use Proposition \ref{Lagrange_Coeffs_Formula} to estimate the size of coefficients first for the class of
strictly positive definite functions developed in \cite{HNW}. Then we attempt to do 
the same for the more general class of kernels of polyharmonic and related type 
of \cite{HNW2}.

\paragraph{Sobolev kernels on compact Riemannian manifolds} We 
begin by considering kernels $k = \kappa_{m}$ with native space inner product given by an expression like
\begin{equation}\label{inner_product}
 \left \langle u,v\right\rangle_{\kappa_m} := 
  \left\langle u,v\right\rangle_{\kappa_m,\M}:= \int_{\M} \beta(u,v)_x \dif \mu 
\end{equation}
where  $\beta$ is a pointwise bilinear form 
 $\beta(u,v)_x := \sum_{j=0}^m c_j\langle\nabla^j u,\nabla^j v\rangle_x$ with
 the condition\footnote{in such cases, the kernel $\kappa_m$ is 
 naturally (strictly) positive definite and moreover, $\kappa_m$  is the fundamental solution 
 to the elliptic operator 
$\L_m=\sum_{j=0}^m c_j (\nabla^j)^*\nabla^j$} that $m > d/2$ and
 $c_0,c_m\ne0$ and  $c_j\ge 0$ for all $j=0,\dots,m$.
 Such kernels were considered in \cite{HNW}  and existence
 was demonstrated for all $d$-dimensional, connected, compact Riemannian manifolds.
 
In this case, the inner product (\ref{inner_product}) is  a Sobolev inner product,
 and it has a natural  
 generalization to inner products for subsets $\Omega$:  namely
 $
 \left \langle u,v\right\rangle_{\kappa_m,\Omega} 
 =\left \langle u,v\right\rangle_{W_2^m(\Omega)}
 =\int_{\Omega} \beta(u,v)_x \dif \mu
 $. 
It is possible to estimate $|\langle \chi_{\xi},  \chi_{\zeta}\rangle_{\kappa_m}|$:
$$|\langle \chi_{\xi},  \chi_{\zeta}\rangle_{k_m}| \le 
\| \chi_{\xi}\|_{W_2^m\bigl(\M\setminus \b(\xi, \frac{\d(\xi,\zeta)}{2})\bigr)} 
\|\chi_{\zeta}\|_{W_2^m(\M)} 
+ 
\| \chi_{\xi} \|_{W_2^m(\M)} 
 \| \chi_{\zeta}\|_{W_2^m\bigl(\M\setminus \b(\zeta, \frac{\d(\xi,\zeta)}{2})\bigr)}  
 $$
 By way of 
\cite[Corollary 4.4]{HNW} we have that
$$|A_{\xi,\zeta}|=|\langle \chi_{\xi},  \chi_{\zeta}\rangle_{\kappa_m}| \le 
C_{m,\M} q^{d-2m} e^{-\nu\frac{\d(x,\xi)}{h}}.$$

Unfortunately, this family of kernels is not suitable for treating practical problems. 
In particular, the kernels having native space inner products of the form (\ref{inner_product}),
even when $\M$ is the sphere, are not known to have closed form representations in terms of
the spatial variable (despite being zonal and  having a simple and satisfying Fourier-Legendre expansion).

To remedy this, we remove the restriction that the coefficients $c_j$ are non-negative (although
$c_m$ must be positive). An immediate consequence of this is that we must contend with a conditionally positive 
definite kernel. The upshot is that, for a large class of interesting manifolds (including spheres
and projective spaces) we can write the Dirichlet form (\ref{inner_product}) 
as linear combinations of powers of the Laplace--Beltrami operators.  The motivation for this approach is that the restricted surface splines are fundamental solutions for operators of this type.
We now describe this.

\paragraph{Polyharmonic and related kernels on 2 point homogeneous spaces}
Let $\M$ be a compact, two point homogeneous space. Included among these are
spheres, $SO(3)$ and various projective spaces.
For our purposes, this is a metric space with distance function $\d(x,y)$ and measure
$\mu$ for which $\mu(\bfb(x,r)) = \mu\{y\mid \d(x,y)\le r\} \sim r^d$.
Because it is compact, there is a 
Laplace--Beltrami operator $\Delta$ with countable spectrum $\sigma(\Delta) = \{\lambda_0, \lambda_1,\dots\}$. Denote the corresponding
orthogonal, $L_2$ normalized eigenfunctions  for 
$\Delta$ by  
$(\psi_j)_{j\in \nats}$.

For such a manifold and for any $k\in \nats$, the operator $(\nabla^k)^*\nabla^k$ can be expressed
as $\sum_{j=0}^k b_{\nu} \Delta^{j}$ with $b_k=(-1)^k$. Consequently, any operator of the form
$\sum_{j=0}^k c_{j} (\nabla^{j})^* \nabla^j$
 can be expressed as $\sum_{j=0}^k b_{j} \Delta^{\nu}$
with $b_k=(-1)^kc_k $
and vice-versa:
\begin{equation}\label{invariance}
\forall  (b_0,\dots b_m)
\  \exists (c_0,\dots,c_m) 
\text{ with }b_m = (-1)^m c_m \text{ and }  \sum_{j=0}^m b_j \Delta^j = \sum_{j=1}^m c_j(\nabla^j)^*\nabla^j. 
\end{equation}

Suppose that the kernel $k_m:\M\times\M\to \reals$  
acts as the Green's function for the
elliptic operator $\L_m :=\sum_{j=0}^m b_j \Delta^j = Q(\Delta)$, in the sense that
$$f = \int_{\M} k_m(\cdot,\alpha) \L_m\bigl[ f(\alpha) -p_f(\alpha)\bigr] \dif \alpha+p_f$$
where $p_f(x)$ is the projection on $\Pi_{\J}$, $p_f = \sum \langle f,\psi_j\rangle_{L_2(\M)} \psi_j $,  
and the complementary part of the spectrum of $\L_m$, 
$\{Q(\lambda_j)\mid j\notin\J\}\subset \sigma(\L_m)$, is
real and lies to one side of $0$ (without loss, we can take $\sigma(\L_m)\subset (0,\infty)$ -- namely, by
considering $-k$ if needed; this is equivalent to taking $b_m >0 $ since the spectrum
of $\Delta$ has $\infty$ as an accumulation point
and $(\lambda_k)^m>| \sum_{j=0}^{m-1} b_j(\lambda_k)^j|$ for all but finitely many $k$). Such a kernel is said to be of {\em polyharmonic or related type}.

\paragraph{The native space ``inner product" on subsets}
It follows directly that $k_m$ is conditionally positive definite with respect to $\Pi_{\J}$. What's more,
when $\L_m \Pi_{\J} = \{0\}$, the native space semi-inner product can be expressed as
$$\langle u,v\rangle_{k_m,\J} = \langle \L_m u, v\rangle_{L_2(\M)}
 = 
 \int_{\M} \beta(u,v)_x \dif \mu(x)$$ 
with $\beta(u,v)_x = \sum_{k=0}^m c_{k}\langle \nabla^{k}u,\nabla^{k}v\rangle_x$ and
$c_0,\dots,c_m$ guaranteed by (\ref{invariance}).
The latter expression
allows us to extend naturally the native space inner product to measurable subsets 
$\Omega$ of $\M$. 
Namely,  
$$\langle u,v\rangle_{\Omega, k_m,\J} :=\int_{\Omega} \beta(u,v)_x \dif \mu(x).$$
This has the desirable property of set additivity: for \ sets
$A$ and $B$ with $\mu(A\cap B) = 0$,
we have $\langle u,v\rangle_{A\cup B, k_m,\J} = \langle u,v\rangle_{A, k_m,\J}+\langle u,v\rangle_{ B, k_m,\J}.$
Unfortunately, since some of the coefficients $c_k$ may be negative, $\beta(u,u)$ and 
$\langle u,u\rangle_{\Omega, k_m,\J}$ may assume negative values for some $u$:  
in other words, the bilinear form
$(u,v)\mapsto \langle u,v\rangle_{\Omega, k_m,\J}$ is only an {\em indefinite} inner product.

However, when $\Omega$ has Lipschitz boundary and $u$ has many zeros, we can relate the 
quadratic form  
$\ns{u}_{\Omega,k_m,\J}^2=\langle u,u\rangle_{\Omega, k_m,\J} $ 
to a Sobolev norm $\|u\|_{W_2^m(\Omega)}^2$. Arguing as in 
\cite[(4.2)]{HNW2}, 
we see that
$$c_m |u|_{W_2^m(\Omega)}^2
 -
\bigl(\max_{j\le m-1}{|c_j|}\bigr)\|u\|_{W_2^{m-1}(\Omega)}^2
\le
 \int_{\Omega}  \beta(u,u)_x \dif \mu (x) 
\le 
\bigl(\max_{j\le m}{|c_j|}\bigr) \|u\|_{W_2^m(\Omega)}^2.
$$
If $u|_{\Xi} = 0$ on a set $\Xi$ with
$h(\Xi,\Omega) \le h_0$ with $h_0$ determined only by the boundary of $\Omega$ (specifically
the radius and aperture of an interior cone condition satisfied by $\partial \Omega$),
Theorem A.11 of \cite{HNW2} guarantees that 
$\|u\|_{W_2^{m-1}(\Omega)}^2 \le Ch^2 |u|_{W_2^m(\Omega)}$
with $C$ depending only on the order $m$, the global geometry of $\M$ and the roughness of the boundary (in this case, depending only on the aperture of the interior cone condition). Thus, by choosing $h$ sufficiently small,
$h\le h^{*}$, where $h^{*}$ satisfies the two conditions 
\begin{equation}\label{hstar}
h^*\le h_0\qquad \text{and}
 \qquad
 C (h^*)^2\times \bigl(\max_{j\le m}{|c_j|}\bigr) \le \frac{|c_m|}{2},
 \end{equation}  
 we have 
$$\frac{c_m}{2}\|u\|_{W_2^m(\Omega)}^2
\le
\ns{u}_{\Omega,k_m,\J}^2
\le 
 \left(\max_{j\le m}|c_j|\right)
 \|u\|_{W_2^m(\Omega)}^2. $$
  
The threshold value $h^*$ depends on the coefficients $c_j$ as well as  the radius $R_{\Omega}$ and aperture $\phi_{\Omega}$ of the cone condition for $\Omega$.
When $\Omega$ is an annulus of sufficiently small inner radius, the cone parameters can be replaced
by a single global constant, and $h_*$ can be taken to depend only on $c_0,\dots, c_m$. In other
words, only on $k_m$ -- cf. \cite[Corollary A.16]{HNW2}. 

A direct consequence of this is positive definiteness for such functions, $\ns{u}_{\Omega,k_m,\J}\ge 0$ with
equality only if $u|_{\Omega}=0$. From this, we have a version of the Cauchy-Schwarz inequality:
if $u$ and $v$ share a set of zeros $Z$ (i.e., $u|_Z = v|_Z = \{0\}$) that is sufficiently dense in $\Omega$, then
\begin{equation}\label{C-S}
\left|\langle u,v\rangle_{\Omega, k_m,\J} \right| \le \ns{u}_{\Omega,k_m,\J}\ns{v}_{\Omega,k_m,\J}
\end{equation}
follows (sufficient density means that $h(Z,\Omega)< h^*$ as above).

\paragraph{Decay of coefficients for kernels of polyharmonic and related type}
Fortunately, Lagrange functions have many zeros, and  \cite[Lemma 5.1]{HNW2}
guarantees that the  Lagrange function $\chi_{\xi}$ satisfies the {\em bulk chasing} estimate
there is a fixed constant $0\le \epsilon<1$ so that for radii $r$ less than a 
constant $\inj$ depending on $\M$ (for a compact, 2-point homogeneous space, the injectivity radius 
is $\inj  = \mathrm{diam}(\M)/2$) the estimate
$
\| \chi_{\xi}\|_{W_2^m(\M\setminus \bfb(\xi,r))} \le
\epsilon \| \chi_{\xi}\|_{W_2^m(\M\setminus \bfb(\xi,r -\frac{h}{4h_0}))}
$
holds.
In other words,  a fraction of $1-\epsilon$ of the bulk of the tail 
$\| \chi_{\xi}\|_{W_2^m(\M\setminus \bfb(\xi,r))}$ is to be found in the annulus 
$ \bfb(\xi,r)\setminus \bfb(\xi,r -\frac{h}{4h_0})$
 of width $\frac{h}{4h_0} \propto h$ 
(with constant of proportionality $\frac{1}{4h_0}$ depending only on $\M$, $m$ and the
boundary of $\Omega$).
Provided $r\le \inj$,  it is possible to iterate this $n$ times for $\frac{nh}{4h_0}\le r$.
It follows that there is $\nu = - 4h_0 \log \epsilon>0$ so that
\begin{equation*}
\| \chi_{\xi}\|_{W_2^m(\M\setminus \bfb(\xi,r))}
\le \epsilon^n \|\chi_{\xi} \|_{W_2^m(\M)}
\le Ce^{-\nu r/h}\|\chi_{\xi} \|_{W_2^m(\M)}.
\end{equation*} 
By \cite[(5.1)]{HNW2} (a simple comparison of $\chi_{\xi}$ to a smooth ``bump'' $\phi_{\xi}$ of radius
$q$ -- also an interpolant to the delta data $(\delta_{\xi})$, but worse in the sense
that $\ns{\chi_{\xi}}_{k_m,\J}\le \ns{\phi}_{k_m,\J}$ -- see the proof of Theorem \ref{main} below)
we have
\begin{equation}
\label{Lag_decay}
\| \chi_{\xi}\|_{W_2^m(\M\setminus \bfb(\xi,r))}
\le
C q^{d/2-m} e^{-\nu \frac{r}{h}}.
\end{equation}
This leads us to our main result.

\begin{theorem}\label{main} 
Let $\M$ be a compact, 2-point homogeneous manifold and let $k_m$ be a kernel
of polyharmonic or related type, so that the associated elliptic operator $\L_m$
annihilates the polynomial space $\Pi_{\J}$. Let $\rho>0$ be a fixed mesh ratio.

There exist constants $h^*$,  $\nu$  and $C$ depending only on $\M$ and $k_m$ 
if $\Xi\subset \M$ is sufficiently dense (i.e., $h(\Xi,\M)\le h^*$)  
then the coefficients of the Lagrange function 
$\chi_{\zeta} = \sum_{\xi\in\Xi} A_{\zeta,\xi}k_m(\cdot,\xi) + p_{\zeta}\in  S(\Xi,\J)$
satisfy
\begin{equation}\label{Coeff}
|A_{\zeta,\xi} |\le C q^{d-2m} \exp{\left(-\nu \frac{\d(\xi,\zeta)}{h}\right)}.
\end{equation}
\end{theorem}
\begin{proof}
By Proposition \ref{Lagrange_Coeffs_Formula} and set additivity, we have that
$$ |A_{\zeta,\xi} | = \langle \chi_{\xi},\chi_{\zeta} \rangle_{ k_m,\J} 
=
 \langle \chi_{\xi},\chi_{\zeta} \rangle_{ \Omega_{\zeta} ,k_m,\J} 
+ \langle \chi_{\xi},\chi_{\zeta} \rangle_{ \Omega_{\xi}, k_m,\J},
$$ 
 where
$\Omega_{\zeta} = \left\{ \alpha \in \M\mid \d(\alpha,\zeta) <  \d(\alpha,\xi)\right\},$ 
$\Omega_{\xi} = \left\{ \alpha \in \M\mid \d(\alpha,\xi) <  \d(\alpha,\zeta)\right\},
$
and (modulo a set of measure zero) 
$\M\setminus\Omega_{\zeta} = \Omega_{\xi}$. For a compact, 2-point homogeneous space, $\Omega_{\zeta}$
and $\Omega_{\xi}$ are two balls of radius $\mathrm{diam}(\M)/2$.

We can apply the Cauchy--Schwarz type inequality (\ref{C-S}) to obtain
\begin{eqnarray*} |A_{\zeta,\xi} | &\le&
\ns{\chi_{\zeta}}_{ \Omega_{\zeta} ,k_m,\J}
 \ns{\chi_{\xi}}_{ \Omega_{\zeta} ,k_m,\J}+
\ns{\chi_{\zeta}}_{ \Omega_{\xi}, k_m,\J} 
\ns{\chi_{\xi}}_{ \Omega_{\xi}, k_m,\J}\\
&\le&
\sqrt{\max_{j\le m}|c_j|}
\left(\|\chi_{\zeta}\|_{ W_2^m(\Omega_{\zeta} )}
 \ns{\chi_{\xi}}_{ \Omega_{\zeta} ,k_m,\J }+
\ns{\chi_{\zeta}}_{ \Omega_{\xi}, k_m,\J } 
\|\chi_{\xi}\|_{ W_2^m(\Omega_{\xi})}\right)
\end{eqnarray*}

Since 
$ \Omega_{\zeta} \subset \b^c(\zeta,r) := \M\setminus \b\left(\zeta,\frac{1}{2} \d(\xi,\zeta)\right) $
and 
$\Omega_{\xi} \subset  \b^c(\xi,r) :=  \M\setminus \b\left(\xi,\frac{1}{2} \d(\xi,\zeta)\right) $, 
we can again employ set additivity and positive definiteness (this time  
$ \ns{\chi_{\xi}}_{ \Omega_{\zeta} ,k_m,\J }\le \ns{\chi_{\xi}}_{ \M ,k_m,\J } $,
which follows from the fact that $\M =  \Omega_{\zeta} \cup  \overline{\Omega_{\xi}}$ and 
that  $\chi_{\xi}$ vanishes to high order in $\Omega_{\xi}$ -- the same holds
for $\chi_{\zeta}$) to obtain
$$ |A_{\zeta,\xi} | \le\sqrt{\max_{j\le m}|c_j|}\left(
\|\chi_{\zeta}\|_{ W_2^m(\b^c(\zeta,r) ) } \ns{\chi_{\xi}}_{ k_m,\J}+
\ns{\chi_{\zeta}}_{  k_m,\J} \|\chi_{\xi}\|_{W_2^m( \b^c(\xi,r)   )}\right).$$

The full energy  of the Lagrange function can be bounded by comparing it to the energy
of a bump  function -- for $\chi_{\xi}$
this is $\phi_{\xi}$, which can be defined on the tangent space by using 
a fixed, smooth, radial cutoff function 
$\sigma$:
$\phi_{\xi}\circ \mathrm{Exp}_{\xi} (x)
= \sigma(|x|/q).$
This is done in \cite[(5.1)]{HNW2} and we have that  
 $\ns{\chi_{\xi}}_{ \M ,k_m,\J}$ and $\ns{\chi_{\zeta}}_{ \M ,k_m,\J}$ are  bounded by $C q^{d/2-m}$.

On the other hand, we can employ (\ref{Lag_decay}) to treat 
$ \|\chi_{\zeta}\|_{ W_2^m(\b^c(\zeta,r) ) }$ and 
$ \|\chi_{\xi}\|_{ W_2^m(\b^c(\zeta,r) ) }$, which gives
$$\| \chi_{\xi} \|_{ W_2^m(\b^c\left(\zeta,r\right)) }, \| \chi_{\zeta} \|_{ W_2^m(\b^c\left(\zeta,r\right)) } \le 
C q^{d/2-m} e^{-\nu \frac{r}{h}} = C q^{d/2-m} e^{-\nu \frac{\d(\xi,\zeta)}{2h}}.$$
From this, the result follows.
\end{proof}
\begin{note} A similar argument shows that, on a compact, 2-point homogeneous manifold, 
the Lagrange function coefficients $A_{\xi,\zeta}$ for a general kernel $k_m$ of polyharmonic and related type 
(regardless of whether $\L_m$ annihilates $\Pi_{\J}$) decay like
\begin{equation}
\label{General_Lag_decay}
|A_{\xi,\zeta}|
\le
C 
q^{d-2m}\max(\exp(-\nu \frac{r}{h}), h^{2m}).
\end{equation}
 In such cases, the theory developed here and in \cite{HNW2} indicate a slower decay for
 Lagrange functions and coefficients (although it remains an open problem to determine if
 these rates can be improved, and by how much).
In particular, this holds for the restricted surface splines $k_m(x,\alpha) = (1-x\cdot\alpha)^{m-d/2}$ on odd dimensional spheres ($d \in 2\nats+1$ and $m>d/2$) as well as the surface splines on SO(3) 
(see Section \ref{S:LagrangeBasis}).
\end{note}
\begin{note} Theorem \ref{main} holds for restricted surface splines $k_m(x,\alpha) = (1-x\cdot \alpha)^{m-d/2} \log (1-x\cdot \alpha)$
on spheres
of even dimension ($d\in 2\nats$ and $m>d/2$) -- in particular for $\mathbb{S}^2$. See Section
\ref{S:Examples} below for some numerical examples in this setting.
\end{note}
%
\section{A better basis: truncating the Lagrange basis}\label{S:better_basis}

We now want to show the \emph{existence} of a good approximation to
the Lagrange function $\chi_\xi$ that uses many fewer elements in its
kernel expansion than the $N$ needed for $\chi_\xi$ itself. To do
this, we will start with the expansion $\chi_\xi = \sum_{\zeta\Xi}
A_{\xi,\zeta} \kappa(\cdot,\zeta)$ and approximate it by a truncated
expansion of the form
\[
\widetilde \chi_\xi = \sum_{\zeta\in\Upsilon(\xi)} A_{\xi,\zeta}
\kappa(\cdot,\zeta) = \chi_\xi- \sum_{\zeta\not\in \Upsilon(\xi)}
A_{\xi,\zeta} \kappa(\cdot,\zeta),
\]
where $\Upsilon(\xi)\subset \Xi\cap B(\xi,r(h))$. Our goal is to show
that, under the assumptions listed below, which apply to a wide class
of kernels, we may take $r(h)=Kh|\log(h)|$, where $K>\frac{2m}{\nu}$,
while maintaining $\|\widetilde \chi_\xi - \chi_\xi\|_\infty \le
Ch^J$, $J:=K\nu -2m$.

How many basis elements are used in expanding $\widetilde \chi_\xi$?
Doing a simple volume estimate shows that the number required is 
\[
\# \Upsilon(\xi) = \calo((Kh|\log h|)^d/q^d) = \calo(|\log h|^d) =
\calo((\log N)^d) \ll N,
\]
where we have used $h/q=\rho$ and $N=\calo(h^{-d})$.

One final remark before proceeding with the analysis. Finding
$\widetilde \chi_\xi$ requires knowing the expansion for $\chi_\xi$
and carrying out the truncation above. This is expensive, although it
does have utility in terms of speeding up evaluations for
interpolation when the same set of centers is to be used
repeatedly. The main point is that we now know roughly how many basis
elements are required to to obtain a good approximation to
$\chi_\xi$. We are currently engaged in investigating cost effective
algorithms to obtain approximate Lagrange functions similar to
$\widetilde \chi_\xi$.

\paragraph{First assumptions}
We make the following three assumptions
\begin{enumerate}
\item The Lagrange functions decay at a rate $|\chi_{\xi}| \le C_L \exp\left(-\nu_L\frac{\d(x,\xi)}{h}\right)$.
\item The kernel coefficients of the Lagrange function decay like 
$|A_{\xi,\zeta}|\le C_c \exp\left(-\nu_c\frac{\d(x,\xi)}{h}\right)$.
\item The Lagrange basis is $L_p$ stable in the sense that
$$
c_1 q^{d/p} \|\bfa\|_{\ell_p(\Xi)}
\le 
\left\|\sum_{\xi\in\Xi} a_{\xi} \chi_{\xi}\right\|_{L_p(\M)}
\le 
c_2 q^{d/p} \|\bfa\|_{\ell_p(\Xi)}.
$$
\end{enumerate}
We note that the family of restricted surface splines on $\mathbb{S}^d$ when $d\in 2\nats$ satisfy these three conditions (conditions 1 and 3 are in \cite{HNW2}, while condition 2 follows from Theorem \ref{main}), as do the
Sobolev splines on any compact Riemannian manifold $\M$ (condition 1 follows from \cite{HNW}, condition
3 from \cite{HNSW} and condition 2 from Theorem \ref{main} again).

\paragraph{Decay} By the estimate of coefficients in Theorem \ref{main}, it suffices to retain only the part of $\Xi$ that is within 
$K h|\log h|$ from $\xi$, since the coefficients we cut out have
size roughly  $Ch^{d-2m} h^{K \nu}$. There are no more than $\#\Xi\le C_{d,\rho} h^{-d}$ of them on the $d$-sphere, and the kernel is uniformly bounded, so we have that 
$$|\widetilde{\chi}_{\xi}(x) - \chi_{\xi}(x)|\le C h^{K\nu - 2m},$$ 
and we should
choose $K>\frac{2m}{\nu}$ at least. 
Indeed, the pointwise estimate above shows that 
$$|\widetilde{\chi}_{\xi}(x)| \le C\left(e^{-\nu \frac{\d(x,\xi)}{h}}+ h^{K\nu-2m}\right)\le C \left(1+\frac{\d(x,\xi)}{h}\right)^{2m-K\nu},$$ 
which indicates that we may 
wish to choose $K$ even larger. This is at our discretion, but to preserve
stability, we choose $K>\frac{2m+d}{\nu}$.

\paragraph{Computational efficiency} Since we retain only the coefficients centered at a distance of 
$\mathcal{O}(h|\log h| )$ from $\xi$, we use 
$$\#\Upsilon(\xi) = \mathcal{O}\left(\frac{\left(h|\log h|\right)^d}{q^d} \right)= \mathcal{O}\bigl((\log N )^d\bigr)$$ 
coefficients (when centers are quasiuniform) to compute each  basis function $\widetilde{\chi}_{\xi}$. 

\paragraph{Stability} By the $L_{\infty}$ stability of the Lagrange basis, 
for $s\in S(k,\Xi)$,
the samples  
$s|_{\Xi} =:(A_{\xi})_{\xi\in\Xi}$ are bounded in the $\ell_{\infty}$ norm  
by $\frac{\|s\|_{\infty}}{c_1} $. Using
 the same coefficients but in the new basis $\tilde{\chi_{\xi}}$,  
we form $\tilde{s} = \sum A_{\xi} \widetilde{\chi}_{\xi}$.
The difference between the original and new function is  
\begin{equation}\label{ell_infty}
\|\tilde{s} - s\|_{\infty} \le \|A\|_{\ell_{\infty}}\sum_{\xi} |\chi_{\xi}(x) - \tilde{\chi_{\xi}}(x)|
\le C\|A\|_{\ell_{\infty}} h^{K\nu - 2m} q^{-d} \le  C h^{K\nu - 2m-d} \|s\|_{\infty},
\end{equation}  
so 
\begin{eqnarray}
 \|\tilde{s} \|_{\infty}&\ge&
\bigl(1 - C h^{K\nu - 2m-d}\bigr) \|s\|_{\infty} 
\ge c_1 \bigl(1 - C h^{K\nu - 2m-d}\bigr)  \|A\|_{\infty}  \\
 \|\tilde{s} \|_{\infty}&\le&  \bigl(1 + C h^{K\nu - 2m-d}\bigr) \|s\|_{\infty}\le c_2 \bigl(1 + C h^{K\nu - 2m-d}\bigr)  \|A\|_{\infty}.
 \end{eqnarray}
This can be viewed in two ways. 
\begin{itemize}
\item Provided $h$ is small enough, the family is a basis. There are $\#\Xi$ elements and they are linearly independent. (In particular, if the function is zero, all the coefficients are zero.)
\item The family $(\widetilde{\chi}_{\xi})$ is stable in $L_{\infty}$, since the map  $(A_{\xi})\mapsto \widetilde{s}$ is boundedly invertible. 
\end{itemize}

We stress that it remains to be determined how actually to compute the basis -- we have simply shown that a preconditioner exists that has complexity $\mathcal{O}\left((\log N)^d \right)$.  

\begin{theorem}\label{better_theorem}
Let $\M$ be a compact, 2-point homogeneous manifold and let $k_m$ be a kernel
of polyharmonic or related type, so that the associated elliptic operator $\L_m$
annihilates the polynomial space $\Pi_{\J}$. Let $\rho>0$ be a fixed mesh ratio.

For sufficiently dense  $\Xi$, with $h = h(\Xi,\M) \le H$, with $H$ a constant depending only
on $\rho$, $\M$ and $k_m$, there is a basis $(b_{\xi})_{\xi\in\Xi}$ 
whereach basis element
$b_{\xi} = \sum_{\zeta \in\Xi} A_{\xi,\zeta} k_m(\cdot,\zeta)$ is  composed of kernels centered
in the ball $B(\xi, -K\log h).$
The following are satisfied:
\begin{itemize}
\item The cost of constructing each $b_{\xi}$ is
$\#\{\zeta\mid A_{\xi,\zeta} \ne 0\} \le \tau \left(\log \#\Xi\right)^d$ with $\tau \le C K^d$.
\item
Each basis element exhibits polynomial decay of degree $J:=K\nu-2m$:
there exists 
 $C$ for which
 $$|b_{\xi}(x)|\le C \left(1+\frac{\d(x,\xi)}{h}\right)^{-J}.$$
\item The basis is $L_p$ stable: there are $c_1,c_2$ for which
$$c_1 q^{d/p} \|\bfa\|_{\ell_p(\Xi)}
\le 
\left\|\sum_{\xi\in\Xi} a_{\xi} b_{\xi}\right\|_{L_p(\M)}
\le 
c_2 q^{d/p} \|\bfa\|_{\ell_p(\Xi)}.$$
\end{itemize}
\end{theorem}
\begin{proof}
It remains to demonstrate the $L_p$ stability of $(b_{\xi} ) = (\tilde{\chi}_{\xi})$ for $1\le p<\infty$. 

When $p=1$, we 
consider  a sequence $\bfa = (a_{\xi})_{\xi\in\Xi}\in \ell_{1}(\Xi)$
and 
set $\tilde{s} := \sum a_{\xi} \tilde{\chi}_{\xi}$.
H{\" o}lder's inequality gives
\begin{equation}\label{ell_one}
\|\tilde{s}  - s \|_{L_1(\M)} \le C \|\bfa\|_{\ell_1(\Xi)}\mathrm{vol}({\M})h^{K\nu -2m}
\end{equation}
since for each $x$ we have the  
estimate
$|\tilde{s}(x)  - s(x)| \le  \left( \sum_{\xi\in\Xi}|a_{\xi}|  \right) \max_{\xi\in\Xi}|\chi_{\xi}(x) - \tilde{\chi}_{\xi}(x)|$. 
Interpolating between (\ref{ell_one}) and (\ref{ell_infty}) (i.e.,  interpolating
the finite rank operator $\bfa \mapsto (s - \tilde{s})$)
gives
$$\|s - \tilde{s}\|_{L_p(\M)} \le C h^{K\nu - 2m -d(1-1/p) }\|\bfa\|_{\ell_p(\Xi)}.$$

Therefore,
$$\|s\|_{L_p(\M)} - C h^{K\nu - 2m-d(1-1/p)} \|\bfa\|_{\ell_p(\Xi)} 
\le \|\tilde{s}\|_{L_p(\M)}
\le
\|s\|_{L_1(\M)} + C h^{K\nu - 2m-d(1-1/p)} \|\bfa\|_{\ell_p(\Xi)} $$
and we have 
$$c_1q^{d/p}\|\bfa\|_{\ell_p(\Xi)} (1- C h^{K\nu - 2m-d} )
\le \|\tilde{s}\|_{L_p(\M)}
\le
c_2q^{d/p}\|a\|_{\ell_p(\Xi)}(1+ C h^{K\nu - 2m-d}).$$
\end{proof}
\begin{note} For a
general kernel $k_m$ of polyharmonic and related type 
(where $\L_m\Pi_{\J} \ne \{0\}$), the estimate for the 
decay of Lagrange function coefficients $A_{\xi,\zeta}$ 
is too slowly to guarantee stability of the truncated ``basis''.
In this case, we can guarantee only that tail of the coefficients
is uniformly bounded, $|A_{\xi,\zeta}| \le C h^d$, and the
best estimate we can give to a truncated Lagrange function $\tilde{\chi}_{\xi}$
is $|\chi_{\xi}(x) - \tilde{\chi}_{\xi}(x)|\le C$.
\end{note}


\section{Numerical examples}\label{S:Examples}
In this section we give some numerical illustrations of the previous results.  In the first example, we provide results for
restricted surface splines on $\mathbb{S}^2$ that support Theorem \ref{main}.  In particular, we demonstrate that the 
constants $C$ and $\nu$, which govern the rate of decay, are in fact quite reasonable.  In the second 
example, we illustrate how the results from Section \ref{S:better_basis} can be used for practical computations 
of surface spline interpolants on $\mathbb{S}^2$ that involve large point sets.  Finally, in the last example, we
investigate the decay rate of the Lagrange coefficients for the restricted surface spline to the Torus, a manifold not covered
by the present theory.

\begin{example}\label{first} 
 \begin{figure}[ht]
\centering
\includegraphics[height=90mm]{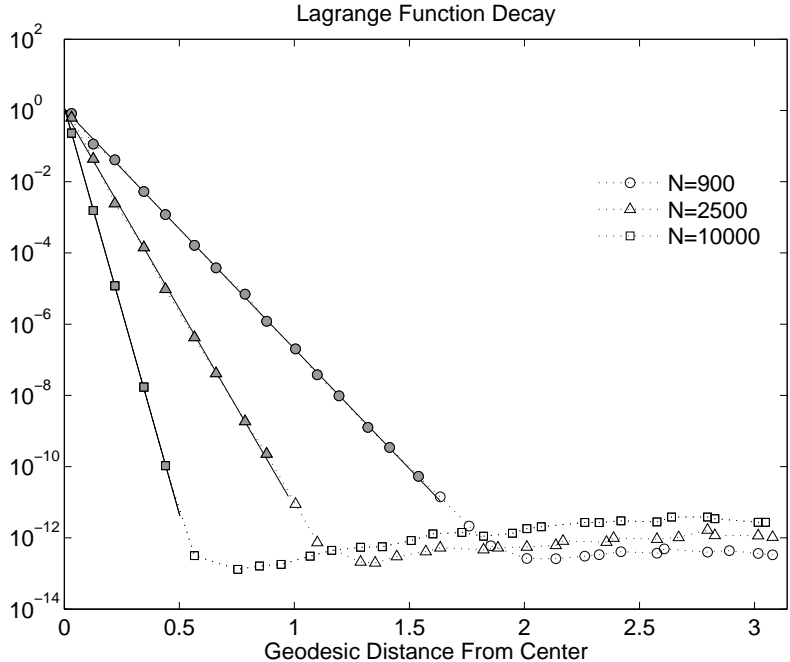}
\caption{Maximum latitudinal 
values of the Lagrange function for 
the kernel $k_2(x,\alpha) = (1-x\cdot\alpha)\log(1-x\cdot\alpha)$. 
This experiment was carried out in double precision arithmetic and
the plateau at roughly $10^{-11}$ occurs due to ill conditioning of
the collocation matrices and truncation error.
}\label{ldTPS}
\end{figure}

We demonstrate the decay of 
Lagrange functions and their coefficients for the second order restricted
surface spline (also known as the thin plate spline) 
$k_2(x,\alpha)  = (1-x\cdot\alpha)\log(1-x\cdot \alpha)$.  The interpolant takes the form $\chi_{\xi} = \sum_{\zeta\in\Xi} A_{\xi,\zeta}k(\cdot,\zeta) +p_{\xi}$, where $p_{\xi}$ is a degree $1$ spherical harmonic.  In this example, we use the ``minimal energy points'' of Womersley for the sphere -- these are described and distributed at the website  \cite{Wom}. The value
of these point sets is as benchmarks. Each set of centers has a nearly identical mesh ratio. Furthermore,
the important geometric properties (e.g., fill distance and separation distance) are explicitly documented. Their potential theoretic properties and their importance in constructing quadrature rules and spherical designs, 
which are discussed in \cite{Saff, SlWom}, are not pertinent to this work.
Because of the nice geometric properties of the minimal energy point sets, it is sufficient to consider the Lagrange function $\chi_{\xi}$ centered at the north pole $\xi = (0,0,1)$. 

\begin{figure}[h]
\centering
\includegraphics[height=90mm]{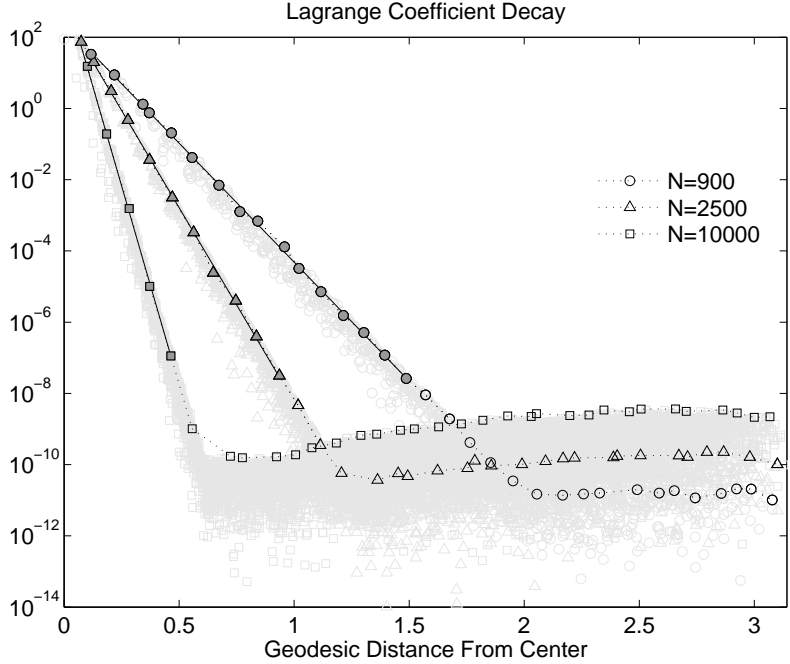}
\caption{Plot of coefficients for a Lagrange function 
in the kernel space $S(k_2,\Xi)$.
This experiment was carried out in double precision
arithmetic.
}\label{lcTPS}
\end{figure}

Figure \ref{ldTPS}  displays the maximal
latitudinal values\footnote{The function $\chi_{\xi}$
is evaluated on a set of points $(\phi,\theta)$ with $n_0$ equispaced latitudes $\phi\in [0,\pi]$
and $n_1$ equispaced longitudes $\theta\in [0,2\pi]$}
 of $\log_{10}|\chi_{\xi}|$.  We clearly observe the exponential decay of the Lagrange function\footnote{At least until a 
terminal value of roughly $10^{-11}$, at which point there is a plateau beyond which the
values no longer decay -- see  below and Figure \ref{lcFREE_TPS} for an explanation of this.} 
$$|\chi_{\xi}(x)|\le C_L \exp\left(-\nu_L \frac{d(x,\xi)}{h}\right)$$
guaranteed by  \cite[Theorem 5.3]{HNW2}.
From this figure, the value of $\nu_L$, which measures the rate of exponential
decay is observed to be close to $1.35$.


We can visualize the decay of the  corresponding coefficients in the same way. We again
take the Lagrange function centered at the north pole:
for each $\zeta' \in\Xi$, the coefficient $|A_{\xi, \zeta'}|$ of the kernel $k(\cdot,\zeta')$ in the expansion
$\chi_{\xi} = \sum A_{\xi,\zeta}k(\cdot,\zeta)+ p_{\xi}$ is
plotted with horizontal coordinate $\sin(\zeta')$.  The results for sets of centers of size 
$N= 900, 2500 $ and $10000$ are given in Figure 2. The  exponential decay
$$|A_{\zeta,\xi} |\le C_c q^{d-2m} \exp{\left(-\nu_c \frac{\d(\xi,\zeta)}{h}\right)}.$$
 guaranteed by Theorem \ref{main} is clearly in force, and we
can estimate the constants $\nu_c$ and $C_c$ for the decay of the coefficients 
following the method used for the Lagrange functions themselves -- we note that coefficients are shifted vertically, which is a consequence of the factor of $q^{d-2m} = q^{-2}$ in the estimate (\ref{Coeff}).  Table \ref{lagrange_stats}   gives more results with some added detail, including  estimates of the constants $C_L$ and $C_c$.

\begin{table}[ht]\label{lagrange_stats}
\begin{center}
\begin{tabular}{|c||c|c|c|c|c|c|}
\hline
$N$ & $h_X$ & $\rho_X$ & $\nu_L$ & $C_L$ & $\nu_c$ & $C_c$\\ 
\hline
\hline
400 & 0.1136 & 1.2930 & 1.1119 & 0.8382  & 1.0997 & 69.9891\\
\hline 
900 & 0.0874  & 1.5302 & 1.3556 & 1.0982& 1.3445 & 231.5573 \\ 
\hline
1600 & 0.0656 & 1.5333 & 1.3513 & 1.2170 & 1.3216 & 324.8534 \\
\hline
2500 & 0.0522 & 1.5278 & 1.3345 & 0.9618& 1.3117 & 470.6483 \\
\hline
5041 & 0.0365 & 1.5304 & 1.3395 & 1.1080  & 1.3158 & 1087.8\\
\hline 
10000 & 0.0260 & 1.5421 & 1.3645 & 1.1934 & 1.3369 & 2564.9\\
\hline
\end{tabular}
\caption{Estimates of fill distance $h$, mesh ratio $\rho$ for
some minimum energy point sets on the sphere 
and $\nu$ and $C$ values for the kernel $k_2(x,\alpha) = (1-x\cdot\alpha)\log(1-x\cdot\alpha)$.}
\end{center}
\end{table}

The perceived plateau present in the Lagrange function values as well as the coefficients
shown in Figures \ref{ldTPS} and \ref{lcTPS} is due purely to round-off error related to the conditioning of
kernel collocation and evaluation matrices.  These results were produced using double-precision (approximately 16 digits) floating point arithmetic.   To illustrate this point, we plot the decay rate of the Lagrange coefficients for the 900 and 1600 point node sets as computed using high-precision (40 digits) floating point arithmetic in Figure \ref{lcFREE_TPS}.  The figure clearly shows the exponential decay does not plateau and continues as the theory predicts.

\begin{figure}[h]
\centering
\includegraphics[height=90mm]{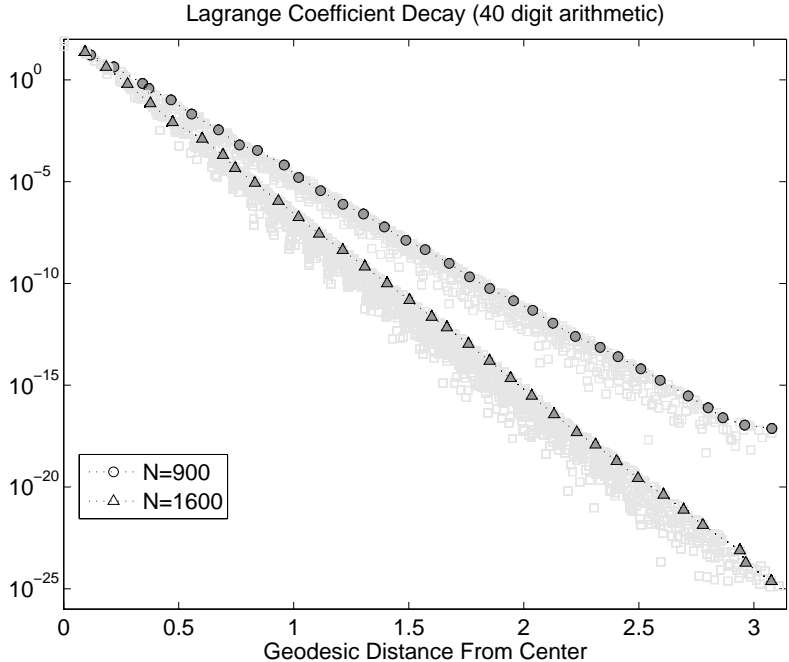}
\caption{Plot of coefficients for a Lagrange function 
in the kernel space $S(k_2,\Xi)$. 
This experiment was carried out in Maple with 40 digit arithmetic.
}\label{lcFREE_TPS}
\end{figure}

\end{example}
\begin{example}
 In this example we construct a basis $(\achi_{\xi})_{\xi\in\Xi}$ for the kernel space $S(k_2,\Xi)\subset C(\sph^2)$ by using 
 $\mathcal{O}\bigl((\log N)^2\bigr)$ centers to construct each $\achi_{\xi}$. 
 
 With this basis, 
 we use an equivalent representation in the form
\begin{equation}
I_{\Xi} f = \sum_{\xi\in \Xi}  c_{\xi}\achi_{\xi}(\cdot), \label{eq:good_basis}
\end{equation}
where each $\achi_{\xi}$ is a local Lagrange function about the node $\xi$ formed by $M \ll N$ basis elements of $S(k_2,\Xi)$.  Specifically, let $\Upsilon(\xi) \subset \Xi$ such that $\xi \in \Upsilon(\xi)$, $\#\Upsilon(\xi)=M$, and $\bigcup_{\xi \in \Xi} \Upsilon(\xi) = \Xi$, then 
\begin{align}
\achi_{\xi} = \sum_{\zeta \in \Upsilon(\xi)} a_{\xi,\zeta} k(\cdot,\zeta) + \sum_{j=1}^4 b_{\xi,j} \varphi_j. \label{eq:approx_lagrange}
\end{align}
The coefficients $a_{\xi,\zeta}$ and $b_{\xi,j}$ are determined from the conditions
\begin{align*}
\achi_{\xi}(\zeta) = 
\begin{cases}
1 & \text{if $\zeta=\xi$}, \\
0 & \text{if $\zeta \in \Upsilon(\xi)\setminus\xi$},
\end{cases}
\quad\text{and}\quad
\sum_{\zeta\in\Upsilon(\xi)} a_{\xi,\zeta} \varphi_j(\zeta)=0.
\end{align*}

The linear system for determining the interpolation coefficients $c_{\xi}$ in \eqref{eq:good_basis} can be written as:
\begin{equation}
\begin{bmatrix}
\coll & \Phi
\end{bmatrix}
\begin{bmatrix}
A_{\Upsilon} \\
B_{\Upsilon}
\end{bmatrix}
\begin{bmatrix}
\vect{c}
\end{bmatrix}
=
\begin{bmatrix}
\vect{f}
\end{bmatrix},
\label{eq:good_lin_sys}
\end{equation}
where $(\coll)_{i,j} = k_2(\xi_i,\xi_j)$ and $B_{i,j} = \phi_j(\xi_i)$, $i,j=1,\ldots,N$. The matrix $A_{\Upsilon}$ is a $N$-by-$N$ \emph{sparse} matrix where each column contains $M$ entries corresponding to the values of $a_{\xi,\zeta}$ in \eqref{eq:approx_lagrange}.  The matrix $B_{\Upsilon}$ is a $4$-by-$N$ matrix with each column containing the values of $b_{\xi,j}$ in \eqref{eq:approx_lagrange}.   With the linear system written in this way, one can view the matrix $[A_{\Upsilon}\; B_{\Upsilon}]^{T}$ as a right preconditioner for the standard kernel interpolation matrix.

If the sets $\Upsilon(\xi)$ are chosen appropriately then the linear system \eqref{eq:good_lin_sys} should be ``numerically nice'' in the sense that the matrix $\coll A_{\Upsilon} + \Phi B_{\Upsilon}$ should have decaying elements from its diagonal and should be well conditioned.  In the example below, each $\Upsilon(\xi)$ is chosen as the $M-1$ nearest nodes to $\xi$.  
Section \ref{S:better_basis} suggests taking $M= \mathcal{O}\bigl((\log N)^2\bigr)$.
Through trial and error we found that choosing $M=7\lceil (\log_{10} N)^2) \rceil$ gave very good results over several decades of $N$.  Each set $\Upsilon(\xi)$ can be determined $O(\log N)$ operations by using a KD-tree algorithm for sorting and searching through the nodes $\Xi$.  The cost for constructing the KD-tree is $O(N (\log N)^2)$.  Thus, constructing all the sets $\Upsilon$ takes
$O(N (\log N)^2)$ operations. 

To solve this linear system we will use the generalized minimum residual method (GMRES)~\cite{SaadSchultz}.  This is a Krylov subspace method which is applicable to non-symmetric linear systems and only requires computing matrix-vector products.   Ideally, there should be a method for computing these matrix vector products in $O(N)$ or $O(N \log N)$ operations to make GMRES more efficient.  Keiner \emph{et. al.} have shown that this can be done in the case of the kernel matrix $K_{\Xi}$ using fast algorithms for spherical Fourier transforms~\cite{Keiner:2006:FSR:1152729.1152732}.  In the results that follow, we have not used this algorithm, but have instead just computed the matrix vector products directly.  We will investigate the use of these fast algorithms in a follow up study.

For the numerical tests we use icosahedral node sets $\Xi\subset\mathbb{S}^2$ of increasing cardinality.  These were chosen because of their popularity in atmospheric fluid dynamics (see, for example,~\cite{Giraldo:1997,StuhnePeltier:1999,Ringler:2000GeodesicGrids,Majewski:2002GME}) where interpolation between node sets is often required.  The values of $f$ were chosen to take on random values from a uniform distribution between $[-1,1]$.  Table \ref{tbl:iterations} displays the number of GMRES iterations to compute an approximate solution to the resulting linear systems \eqref{eq:good_lin_sys}.  As we can see, the number of iterations stays relatively constant as $N$ increases and does not appear to increase with $N$.  

\begin{table}[h]
\centering
\begin{tabular}{|c|c||c|c|}
\hline
& & \multicolumn{2}{c|}{Number GMRES iterations} \\
$N$ & $m$ &  $tol=10^{-6}$ & $tol=10^{-8}$ \\
\hline
\hline
2562 & 84 & 7 & 5 \\
10242 & 119 & 5 & 7 \\
23042 & 140 & 6 & 7 \\
40962 & 154 & 5 & 7 \\
92162 & 175 & 6 & 8 \\
163842 & 196 & 5 & 7\\
\hline
\end{tabular}
\caption{Number of GMRES iterations required for computing an approximate solution to \eqref{eq:good_lin_sys} using  icosahedral node sets of cardinality $N$.  $m$ corresponds to the number of nodes used to construct the local basis and $tol$ refers to the tolerance on the relative residual in the GMRES method.  The right hand side was set to random values uniformly distributed between $[-1,1]$ and the initial guess for GMRES was set equal to the function values.\label{tbl:iterations}}
\end{table}
\end{example}

\begin{example}
A second example shows similar results for Lagrange functions for the kernel
$k(x,\alpha) = (x - \alpha)^2 \log (x - \alpha)$
restricted to a torus of outer radius $4$ and inner radius $2$.
In other words,  the surface parametrized by
\begin{eqnarray*}
x&=&(3+\cos v)\cos u\\
y&=&(3+\cos v)\sin u\\
z&=&\sin v,
\end{eqnarray*}
with $u,v\in [0,2\pi]$.
This combination of kernel and  manifold is not treated in \cite{HNW2} (the curved
torus is not even a symmetric space) although it is considered in \cite{FuWr},
as a subset of $\reals^3$ that is $\Pi_1$ unisolvent. Indeed the torus is the zero
set of a degree $4$ polynomial in $\reals^d$, a sufficiently dense subset will also be
$\Pi_1$ unisolvent. Hence, interpolation on such sets by shifts of $k$ is well posed.

For this experiment, we consider ``minimum energy'' point sets $\Xi$ 
produced by
Ayla Gafni, Doug Hardin and Ed Saff which have previously been used in 
\cite{FuWr}. In
each case we fix the point $\xi = (4,0,0)$ and 
we examine coefficients $(A_{\xi,\zeta})_{\zeta\in\Xi}$ of the Lagrange function 
$\chi_{\xi} = \sum_{\zeta \in\Xi} A_{\xi,\zeta} k(\cdot,\zeta) +p$.
  To solve the interpolation problem, we add a linear polynomial, $p\in \Pi_1(\reals^3)$, 
  and require  the coefficients to satisfy the
side conditions
 $\sum_{\zeta\in\Xi} A_{\xi,\zeta} q(\zeta)=0,\, \forall q\in\Pi_1(\reals^3)$ -- in other words, the
 coefficients annihilate 
linear polynomials.

Despite the fact that Theorem \ref{main} does not apply, a certain exponential decay is observed for these coefficients as well. As before, the rate of decay seems to be independent of $h$  as well as $N= \#\Xi$.

We provide two ways to visualize the decay of the coefficients: first by arranging them latitudinally with horizontal axis representing the distance from $\xi$ in the $u$ direction and then longitudinally with horizontal axis  representing the distance from $\xi$ in the $u$ direction. For the specific choice of $\xi=(4,0,0)$, both correspond to geodesic distances. (Other choices of $\xi$ are observed to have the same rate of coefficient decay.) 

\begin{figure}[h]
\centering
\begin{tabular}{cc}
\includegraphics[height=60mm]{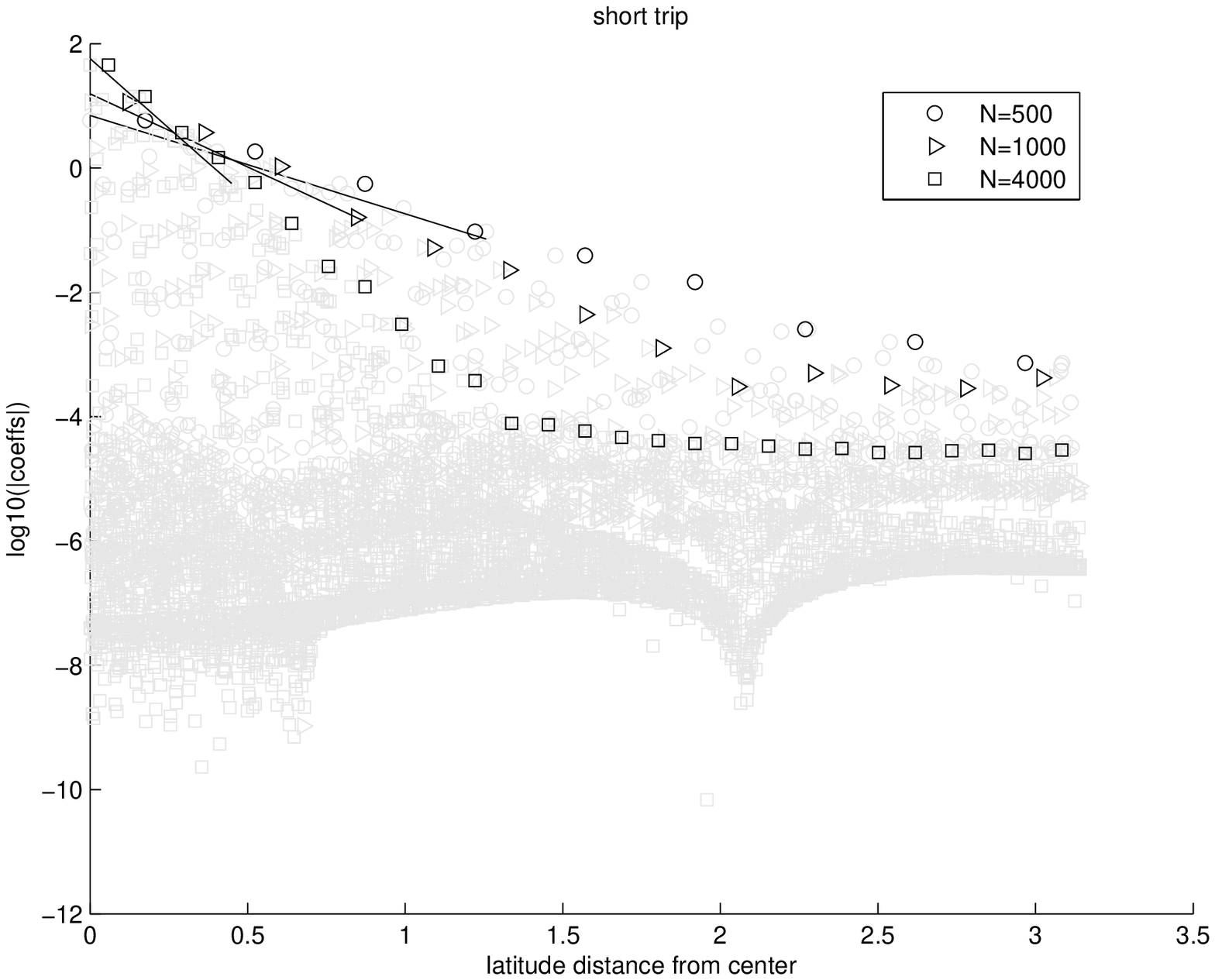}&
\includegraphics[height=60mm]{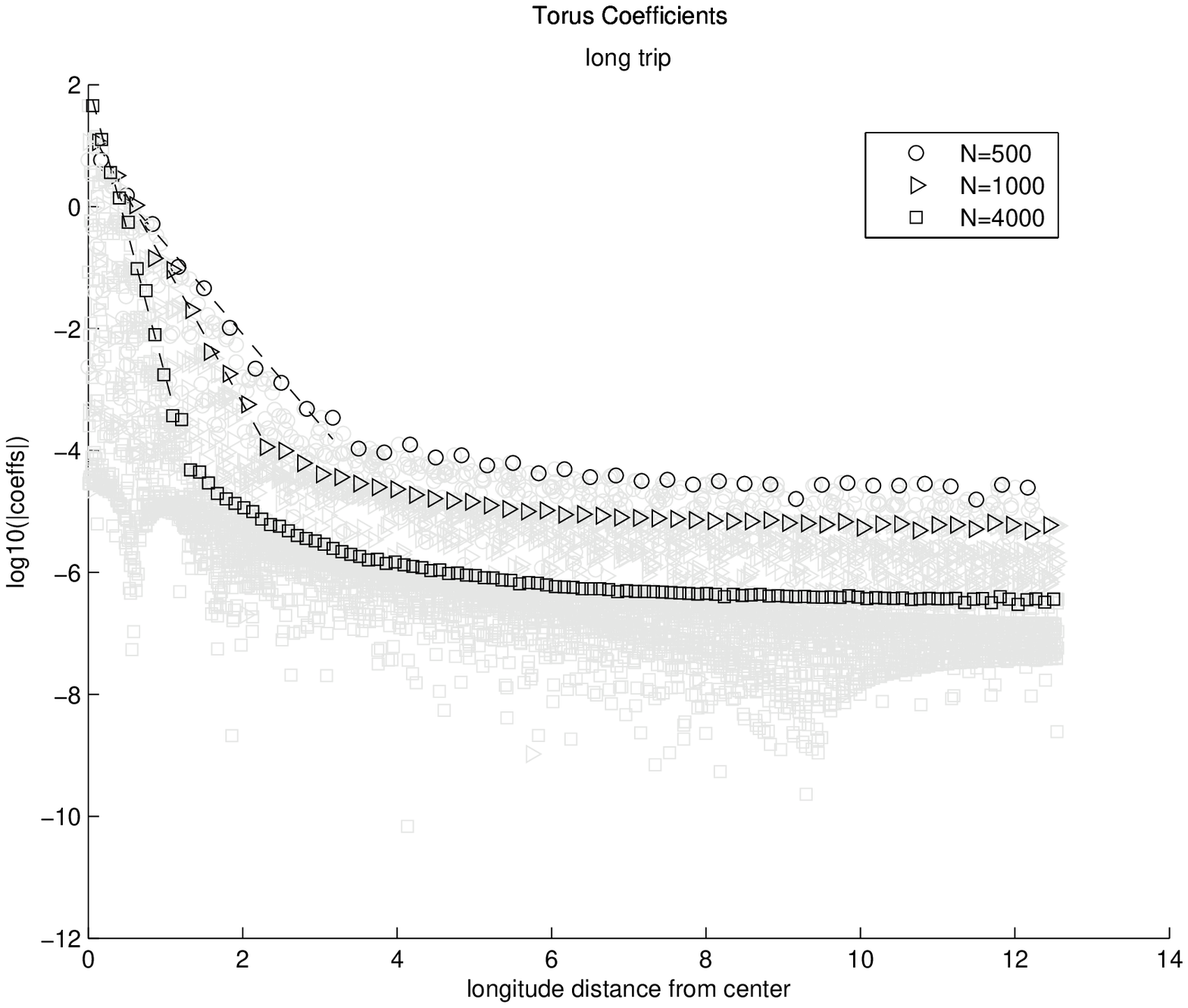}\\
\end{tabular}
\caption{Log plot of  coefficients (in absolute value) of Lagrange functions centered at $(4,0,0)$
for the kernel $k(x,\alpha) = |x-\alpha|^2\log|x-\alpha|$  using minimal energy 
sets on the torus of cardinality on 500, 1000 and 4000.
}\label{Torus}
\end{figure}

A more complete account of this experiment is given in Table \ref{tbl:torus}, where estimates 
of the exponential  decay rates  $\nu_{Lat}$ and $\nu_{Long}$ (in the $v$ and $u$ directions,
respectively) are given. We note that
they are between 1 and $1.3$ in either direction.

\begin{table}[h]
\centering
\begin{tabular}{|c||c|c|c|c|c|c|}
\hline
$N$ & $h_X$ & $\rho_X$ & $\nu_{Lat}$ & $C_{Lat}$ & $\nu_{Long}$ & $C_{Long}$ \\ 
\hline
\hline
500 & 0.3383 & 1.5242 & 1.2312 & 6.9946 & 1.1504 & 7.3231 \\
\hline 
750 & 0.2737  & 1.5024 & 1.1556 & 9.4737 & 1.2122 & 15.4474 \\ 
\hline
1000 & 0.2375 & 1.5014 & 1.2870 & 15.6376 & 1.2401 & 20.4220 \\
\hline
1999 & 0.1639 & 1.4725 & 1.1213& 25.3917& 1.2719 & 49.0776 \\
\hline
3000 & 0.1333 & 1.4479 & 1.0793 &36.1593 & 1.2421 & 58.7687 \\
\hline 
4000 & 0.1151 & 1.4498 & 1.1836 & 57.4460 & 1.2738 & 105.2720\\
\hline
\end{tabular}
\caption{
Results for a Lagrange function experiment on the torus using
the thin plate spline kernel $k(x,\alpha) = |x-\alpha|^2\log|x-\alpha|$.
In addition to fill distance $h$, mesh ratio $\rho$ for minimum energy set of Saff and Hardin, 
estimates of the latitudinal and  longitudinal  exponential decay rate and constant $\nu$ and $C$ values are given.
\label{tbl:torus}}
\end{table}

\end{example}
\bibliographystyle{siam}
\bibliography{Kernel_Coefficients}
\end{document}